\documentclass[a4paper]{amsart}%
\usepackage{amscd}
\usepackage{amsmath}
\usepackage{latexsym}
\usepackage{graphicx}
\usepackage{amsfonts}
\usepackage{amssymb}%
\providecommand{\U}[1]{\protect\rule{.1in}{.1in}}
\input xypic
\xyoption{all}
\CompileMatrices
\providecommand{\U}[1]{\protect\rule{.1in}{.1in}}
\providecommand{\U}[1]{\protect\rule{.1in}{.1in}}
\providecommand{\U}[1]{\protect\rule{.1in}{.1in}}
\nonstopmode
\newcounter{fig}

\numberwithin{equation}{section}

\DeclareMathOperator*{\colim}{colim}

\newcommand{\Z}{\mathbb{Z}}

\theoremstyle{plain}
\newtheorem{Theorem}{Theorem}[section]

\newtheorem{theorem}[Theorem]{Theorem}

\newtheorem{corollary}[Theorem]{Corollary}
\newtheorem{proposition}[Theorem]{Proposition}
\newtheorem{lemma}[Theorem]{Lemma}
\newtheorem{Alternative Version}{Alternative Version}
\theoremstyle{definition}

\newtheorem{definition}[Theorem]{Definition}

\newtheorem{Preliminaries}[Theorem]{}
\newtheorem{remark}[Theorem]{Remark}

\newtheorem{Remarks}[Theorem]{Remarks}

\def\Changed/{\ifvmode\else\vadjust{\vbox
to 0pt{\vskip -\baselineskip\hbox to 0pt{\hss\vrule height 0pt depth 1.2\baselineskip\hskip 1em}\vss}}\fi}
\def\Math#1{\def\MathString{#1}\futurelet\MathDelim\MathChoose}
\def\MathChoose{\ifmmode\let\MathDo\MathString  \else\let\MathDo\MathSkip\fi \MathDo}
\def\MathSkip{\ifx\MathDelim/\def\MathDo{$\MathString$\EatOne} \else\def\MathDo{$\MathString$}\fi\MathDo}
\def\Text#1{\def\TextString{#1}\futurelet\TextDelim\TextSkip}
\def\TextSkip{\ifx\TextDelim/\def\TextDo{\TextString\EatOne}\else\let\TextDo\TextString\fi\TextDo}
\def\EatOne#1{}
\def\SkipToEndScan#1\EndScan{}
\def\Scan#1#2#3{\ifx#1#2#3\expandafter\SkipToEndScan\fi\Scan#1}
\def\Upper#1{\Scan#1aAbBcCdDeEfFgGhHiIjJkKlLmMnNoOpPqQrRsStTuUvVwWxXyYzZ#1#1\EndScan}
\def\Phrase#1 #2/#3/#4=#5 #6/#7/#8.{\expandafter\edef\csname#2#3\endcsname{\noexpand\Text{#6#7}}
\expandafter\edef\csname\Upper#2#3\endcsname{\noexpand\Text{\Upper#6#7}}
\expandafter\edef\csname#1#2#3\endcsname{\noexpand\Text{#5 #6#7}}
\expandafter\edef\csname\Upper#1#2#3\endcsname{\noexpand\Text{\Upper#5 #6#7}}
\expandafter\edef\csname#2#4\endcsname{\noexpand\Text{#6#8}}
\expandafter\edef\csname\Upper#2#4\endcsname{\noexpand\Text{\Upper#6#8}}
}

\begin{document}
\title[Hermitian $K$-theory]{The homotopy fixed point theorem and the Quillen-Lichtenbaum conjecture in
Hermitian $K$-theory}
\author{A.\thinspace J. Berrick, M. Karoubi, M. Schlichting, P.\thinspace A.
{\O }stv{\ae }r}

\begin{abstract}
We settle two conjectures for computing higher Grothendieck-Witt groups (also
known as Hermitian $K$-groups) of noetherian schemes $X$,\textsf{ }under some
mild conditions. It is shown that the comparison map from the Hermitian
$K$-theory of $X$ to the homotopy fixed points of $K$-theory under the natural
$\mathbb{Z}/2$-action is a $2$-adic equivalence. We also prove that the mod
$2^{\nu}$ comparison map between the Hermitian $K$-theory of $X$ and its
\'{e}tale version is an isomorphism on homotopy groups in the same range as
for the Quillen-Lichtenbaum conjecture in $K$-theory. Applications compute
higher Grothendieck-Witt groups of complex algebraic varieties and rings of
$2$-integers in number fields, and hence values of Dedekind zeta-functions.

\end{abstract}
\maketitle

\section{Introduction}

In geometric applications, real topological $K$-groups often yield stronger
results than the more easily computable complex topological $K$-groups. This
is exemplified by Adams' solution of the vector field problem on spheres, and
the image of the $J$-homomorphism in the stable homotopy groups. One can,
however, compute the real topological $K$-groups by using the action of the
group $C_{2}$ of order $2$ on the groups and spaces underlying the complex
theory. More precisely, taking fixed points of the conjugation action yields
an inclusion $O=U^{C_{2}} \subset U$ of the orthogonal group into the unitary
group, and there is an induced homotopy equivalence
\begin{equation}
\label{eqn:ClassicalBOBU}BO \overset{\sim}{\to} BU^{hC_{2}}%
\end{equation}
between the classifying space of $O$ and the homotopy fixed points of $BU$.
It leads to the homotopy fixed point spectral sequence relating the complex
and real $K$-groups
\begin{equation}
\label{eqn:ClassicalOUSpseq}H^{-p}(C_{2};\pi_{q}(BU)) \Rightarrow\pi
_{p+q}(BO).
\end{equation}
\vspace{1ex}

The algebraic analogs of complex topological $K$-groups are Quillen's
algebraic $K$-groups $K_{i}(R) = \pi_{i} BGl(R)^{+}$ of a ring $R$, $i\geq1$.
In motivic lingo, the complex realization functor from $\mathbb{A}^{1}%
$-homotopy theory to the ordinary homotopy category sends algebraic $K$-theory
to complex topological $K$-theory. Similarly, the algebraic analogs of real
topological $K$-groups are the second author's Hermitian $K$-groups $GW_{i}(R)
= \pi_{i} BO(R)^{+}$, $i\geq1$. The Hermitian $K$-groups (also called
\emph{higher Grothendieck-Witt groups}) were introduced in
\cite{Karoubi:batelle} at the same time as algebraic $K$-theory and extended
to schemes in \cite{Sch10(myMV)}. Hermitian $K$-theory yields real topological
$K$-theory via complex realization; see \cite{TripathiSchlichting}. As in
topology, higher Grothendieck-Witt groups often yield stronger results than
algebraic $K$-groups in applications, as exemplified by recent work on
projective modules over smooth affine algebras \cite{AsokFasel},
\cite{FaselRaoSwan}. Another major motivation for computing Hermitian
$K$-groups is the current quest for understanding motivic stable homotopy
groups, a fundamental problem drawing inspiration from topology. \vspace{1ex}

From the homological point of view, Borel's work
\cite{Borel} put on the same footing the general linear group and other
classical groups like the orthogonal or symplectic ones. Rationally, or more
generally up to $2$-torsion, the groups $GW_{i}(R)$, which deal with the
orthogonal and symplectic groups, are well understood thanks to the
\textquotedblleft fundamental theorem of Hermitian $K$%
-theory\textquotedblright\ \cite{Kar1980}, \cite{Sch(myderived)}: they can be
computed in terms of Witt groups and the symmetric part of Quillen's
$K$-theory under the involution induced by the duality functor; see Remark
\ref{rem:oddprime}. On the other hand, much less is known about the
$2$-torsion in $GW_{i}(R)$, and our article provides new tools for computing
these.
\vspace{1ex}

In this paper, under some mild assumptions, we settle two conjectures for
computing Hermitian $K$-groups of commutative rings and more generally of
schemes.\footnote{The results of this paper were found independently by the
third author in the general case \cite{Sch11} and the other authors in the
case of schemes of characteristic 0.} The first conjecture is the algebraic
analog of the homotopy equivalence (\ref{eqn:ClassicalBOBU}) and was
formulated by Thomason in \cite{Thomasonhtylimit} as a homotopy limit problem.
It was explicitly conjectured by Williams \cite[p.\thinspace667]{Wil05} in
relation with geometric topology. \vspace{1ex}

We prove Williams' conjecture in Theorem 2.4. For commutative rings in which
$-1$ is a sum of squares and mild finiteness conditions hold, we obtain a
homotopy equivalence
\begin{equation}
\label{eqn:OrthGpVersionHFT}BO(R)^{+} \sim\left( B\mathrm{Gl}(R)^{+}\right)
^{hC_{2}}%
\end{equation}
valid up to connected components, and an associated spectral sequence
\begin{equation}
\label{eqn:algSpSeqGlO}H^{-p}(C_{2};K_{q}(R)) \Rightarrow GW_{p+q}(R).
\end{equation}
\label{thm:IntegralHtpyLimit} When $-1$ is not a sum of squares in $R$, the
homotopy equivalence (\ref{eqn:OrthGpVersionHFT}) is not valid though its
$2$-adic version still holds (Theorem 2.2). In fact, we prove our results for
schemes - a generality imposed upon us by our use of Nisnevich descent.

Theorems \ref{thm:Z2htpylimit} and \ref{thm:IntegralHtpyLimit} below solve
Williams' conjecture for noetherian schemes.
Our proof uses among other things a result of Hu-Kriz-Ormsby \cite{HKO} which
in turn relies on the solution of Milnor's conjecture by Voevodsky \cite{VV}.
From our solution of Williams' conjecture, we prove general theorems for
higher Grothendieck-Witt groups from their $K$-theory counterparts. As an
example of application, we give a conceptual computation of the Hermitian
$K$-groups of rings of $2$-integers in certain totally real number fields
\cite{BKO1} and relate their orders to values of Dedekind zeta-functions; see
Theorems \ref{thm:ComputationOF} and \ref{thm:zetavalues}. \vspace{1ex}

Our innovation in the proof of Williams' conjecture is the simultaneous proof
of another conjecture: the counterpart, for Hermitian $K$-theory, of the
Quillen-Lichtenbaum conjecture in $K$-theory.
The main goal is to compare the higher Grothendieck-Witt groups with mod
$2^{\nu}$ coefficients to their \'{e}tale analogs. In Theorem \ref{thm:QL}, we
show that the \'{e}tale comparison map for Hermitian $K$-theory is an
isomorphism on homotopy groups in the same range and under the same hypotheses
as it is for $K$-theory. The Hermitian Quillen-Lichtenbaum conjecture was
first explored in \cite{BKO}, where the \'{e}tale comparison map was shown to
be split surjective, and conjectured to be bijective, in sufficiently high
degrees. \vspace{1ex}

Since the work of Artin and Grothendieck, it is well known that for complex
algebraic varieties, \'{e}tale homotopy (with finite coefficients) coincides
with the topological analog. Therefore, our results enable us to compute
higher Grothendieck-Witt groups of complex algebraic varieties in terms of
topological data.
We also give similar computations for totally imaginary number fields in terms
of \'{e}tale cohomology. See Theorems \ref{thm:GWalgKOtop} and
\ref{thm:2cd2computation} for precise statements.

\section{Statement of results}

\label{section:Statementofresults}

Here is a more detailed description of the results in this paper. Most of our
arguments take place in the setting of spectra associated to what we shall
call a \emph{QL} scheme (in honor of Quillen and Lichtenbaum). We let
$\mathrm{vcd}_{2}(X)$ be shorthand for $\sup\{\mathrm{vcd}_{2}(k(x))\mid x\in
X\}$, where, for any field $k$, the virtual mod-$2$ cohomological dimension
$\mathrm{vcd}_{2}(k)$ is the mod-$2$ \'etale cohomological dimension of
$k(\sqrt{-1})$.

\begin{definition}
\label{defn: QL}We call a scheme $X$ a \emph{QL}\textrm{ }scheme if it is
noetherian of finite Krull dimension, $\frac{1}{2}\in\Gamma(X,\mathcal{O}%
_{X})$, $\mathrm{vcd}_{2}(X)<\infty$ and $X$ has an ample family of line
bundles.\medskip
\end{definition}

Note that $\mathrm{vcd}_{2}(X)<\infty$ if $X$ is of finite type over
$\mathbb{Z}[\frac{1}{2}]$ or $X=\mathrm{Spec}(k)$, where $k$ is a field for
which $\mathrm{vcd}_{2}(k)<\infty$. The ampleness condition means that $X$ is
a finite union of open affine subsets of the form $\{f_{i}\neq0\}$ with
$f_{i}$ a section of a line bundle $\mathcal{L}_{i}$ on $X$. Examples include
all affine schemes, separated regular noetherian schemes, and quasi-projective
schemes over a scheme with an ample family of line bundles. Every $QL$ scheme
$X$ is quasi-separated because the underlying topological space of $X$ is
noetherian. We are relieved of any regularity conditions in Definition
\ref{defn: QL} because the descent results in \cite{Sch(myderived)},
\cite{TT90}, and the use of Gabber rigidity for $K$-theory and
Grothendieck-Witt theory, see Theorem \ref{rigidity}, apply to $QL$ schemes.

For a fixed line bundle $\mathcal{L}$ on $X$, let $GW^{\left[  n\right]
}(X,\mathcal{L})$ denote the Grothendieck-Witt spectrum of $X$ with
coefficients in the $n\,$th shifted chain complex $\mathcal{L}[n]$. This is
the Grothendieck-Witt spectrum of the category of bounded chain complexes of
vector bundles over $X$ equipped with the duality functor $E\mapsto
\mathrm{Hom}(E,\mathcal{L}[n])$ and quasi-isomorphisms as the weak
equivalences \cite[\S \ 8]{Sch10(myMV)}. If $X=\mathrm{Spec}(R)$ is affine,
$n=0$ or $2$ and $\mathcal{L}=\mathcal{O}_{X}$, its nonnegative homotopy
groups coincide with Karoubi's Hermitian $K$-groups of $R$ \cite{Kar1980},
with the sign of symmetry $\varepsilon=\pm1$. If $n=1$ or $3$, we recover the
so-called $U$-groups \cite{Kar1980}. For $GW^{\left[  n\right]  }%
(X,\mathcal{L})$, we employ the delooping constructed in \cite[Theorem 5.5 and
Proposition 5.6]{Sch(myderived)}, whose $i$\thinspace th homotopy group
$GW_{i}^{\left[  n\right]  }(X,\mathcal{L})$ is naturally isomorphic to
Balmer's triangular Witt group $W^{n-i}(X,\mathcal{L})$ when $i<0$; see
\cite{Balmer}, \cite[Proposition 6.3]{Sch(myderived)}.

We write $K^{[n]}(X,\mathcal{L})$ for the connective $K$-theory spectrum
$K(X)$ of $X$ equipped with the $C_{2}=\mathbb{Z}/2$-action induced by the
duality functor $\mathrm{Hom}(-,\,\mathcal{L}[n])$. Recall from \cite{Kob99},
\cite[\S \ 7.2]{Sch(myderived)} the natural map
\begin{equation}
GW^{\left[  n\right]  }(X,\mathcal{L})\longrightarrow K^{[n]}(X,\mathcal{L}%
)^{hC_{2}} \label{WilliamsMap}%
\end{equation}
between Hermitian $K$-theory and the homotopy fixed points of $K$-theory.

Throughout the paper we use the term \textquotedblleft
equivalence\textquotedblright\ as shorthand for a \textquotedblleft map that
induces isomorphisms on all homotopy groups.\textquotedblright

\begin{theorem}
[Homotopy Fixed Point Theorem]\label{thm:Z2htpylimit} Let $X$ be a QL scheme
as in Definition \ref{defn: QL}\textsf{ }above. Then for all $\nu\geq1$, the
map (\ref{WilliamsMap}) induces an equivalence of spectra mod $2^{\nu}$:
\[
GW^{\left[  n\right]  }(X,\mathcal{L};\,\mathbb{Z}/2^{\nu})\overset{\simeq
}{\longrightarrow}K^{[n]}(X,\mathcal{L};\,\mathbb{Z}/2^{\nu})^{hC_{2}}.
\]

\end{theorem}

\begin{remark}
Williams conjectured this theorem in \cite[p.\thinspace627]{Wil05} for affine
$X$ (and noncommutative rings), but with no restriction on the cohomological
dimension. In that generality, however, there are counterexamples; see
\cite{HKO} for fields of infinite virtual mod-$2$ cohomological dimension and
\cite[Appendix C]{BKO1} for noncommutative rings.
\end{remark}

Most of the results of this paper deal with $2$-primary coefficients. For
$p$-primary coefficients with $p$ an odd prime see Remark \ref{rem:oddprime}
below. One exception is the following result, which Proposition
\ref{prop:IntegralConverse} shows is the best that we may expect integrally.

\begin{theorem}
[Integral Homotopy Fixed Point Theorem]\label{thm:IntegralHtpyLimit} Let $X$
be a QL scheme. If $-1$ is a sum of squares in all residue fields of $X$, then
the map (\ref{WilliamsMap}) is an equivalence of spectra
\[
GW^{\left[  n\right]  }(X,\mathcal{L})\overset{\simeq}{\longrightarrow}%
K^{[n]}(X,\mathcal{L})^{hC_{2}}.
\]

\end{theorem}

For example, the map (\ref{WilliamsMap}) is an equivalence when $X$ is of
finite type over the Gaussian $2$-integers $\mathbb{Z}[\frac{1}{2},\sqrt{-1}]$
or when $X$ is defined over a field that is not formally real, \textsl{e.g.},
an algebraically closed field of characteristic $\neq2$ or a field
of\textsf{\ }odd characteristic. If $\mathcal{L}=\mathcal{O}_{X}$ then the
converse holds; see Proposition \ref{prop:IntegralConverse}. For example, the
map (\ref{WilliamsMap}) is not an integral equivalence for $X=\text{Spec}(R)$
where $R=\mathbb{Z}[\frac{1}{2}]$, $\mathbb{Q}$ or $\mathbb{R}$.

\medskip

Recall from \cite[Definition 7.1]{Sch(myderived)} (for affine $X$; see also
\cite{Ranicki:book}) the $L$-theory spectrum $L(X,\mathcal{L})$ of a
$\mathbb{Z}[\frac{1}{2}]$-scheme $X$ with coefficients in the line-bundle
$\mathcal{L}$. By \cite[Proposition 7.2]{Sch(myderived)}, its homotopy groups
$\pi_{i}L(X,\mathcal{L})$ are naturally isomorphic to the higher Witt-groups
$W^{-i}(X,\mathcal{L})$ of Balmer \cite{Balmer}. Further, denote by
$\widehat{H}(C_{2},F)$ the Tate-spectrum of a spectrum $F$ with $C_{2}$-action.

\begin{corollary}
\label{cor:LthTate} For any QL scheme $X$, the map%
\[
L(X,\mathcal{L})\longrightarrow\widehat{H}(C_{2},\,K(X,\mathcal{L}))
\]
is a $2$-adic equivalence. It is an integral equivalence under the further
hypothesis of Theorem \ref{thm:IntegralHtpyLimit}.
\end{corollary}

In the formulation of the next theorem we employ the \textquotedblleft
non-connective\textquotedblright\ versions $\mathbb{G}W$ of $GW$
\cite[p.\thinspace430, Definition 8]{Sch10(myMV)}, \cite[Definition 8.6 and
Remark 8.8]{Sch(myderived)} and $\mathbb{K}$ of $K$ \cite[p.\thinspace360,
Definition 6.4]{TT90}, \cite[p.\thinspace123, Definition 12.1]{Sch06(myMathZ)}%
. We note the following consequence of our previous results.

\begin{corollary}
\label{rem:ConnectedVsNonconnected} Theorems \ref{thm:Z2htpylimit} and
\ref{thm:IntegralHtpyLimit} remain valid if one replaces $GW$ and $K$ with
$\mathbb{G}W$ and $\mathbb{K}$, respectively.
\end{corollary}

We write $GW^{\left[  n\right]  }(X_{\mathrm{\acute{e}t}},\mathcal{L})$ for
the value at $X$ of a globally fibrant replacement of $GW^{\left[  n\right]
}(\phantom{X},\mathcal{L})$ on the small \'{e}tale site $X_{\mathrm{\acute
{e}t}}$ of $X$; see \cite{BKO}, \cite{Jardine:simplSpectra}, or
\ref{subsec:PshSpXet} below. Also, recall that a map is said to be
$m$-coconnected when its homotopy fiber is; equivalently, on the $i\,$th
homotopy groups the map induces an isomorphism whenever $i>m$ and a
monomorphism when $i=m$.

\begin{theorem}
[Hermitian Quillen-Lichtenbaum]\label{thm:QL} Let $X$ be a QL scheme. Then for
all $\nu\geq1$ the natural map
\begin{equation}
GW^{\left[  n\right]  }(X,\mathcal{L};\,\mathbb{Z}/2^{\nu})\longrightarrow
GW^{\left[  n\right]  }(X_{\mathrm{\acute{e}t}},\mathcal{L};\,\mathbb{Z}%
/2^{\nu}) \label{eqn:QL}%
\end{equation}
is (\textit{$\mathrm{vcd}$}$_{2}(X)-2$)-coconnected.
\end{theorem}

Theorem \ref{thm:QL} is the evident analog for Hermitian $K$-theory of the
well known $K$-theoretic Quillen-Lichtenbaum conjecture. In \cite{RO05}, the
$K$-theory analog was proved in essentially the same generality as in Theorem
\ref{thm:QL} above; see also Theorem \ref{thm:KQL} below. The statement of
Theorem \ref{thm:QL} was conjectured in \cite{BKO}, where the map was shown to
be split surjective in sufficiently high degrees.



\subsection*{Acknowledgments}

The first author acknowledges NUS research grant R 146-000-137-112 and
Singapore Ministry of Education grant MOE2010-T2-2-125. The third author
acknowledges NSF research grant DMS ID 0906290. He would like to thank Thomas
Unger for useful conversations related to this work. The fourth author
acknowledges RCN research grant 185335/V30.\medskip

\section{Preliminaries\label{Section 2 Preliminaries}}

In this section, we collect a few well-known facts; no originality is claimed.

For a given scheme $X$, fix a line bundle $\mathcal{L}$ on $X$ and set
$\ell=2^{\nu}$. For legibility we often write $GW(X)$ for the spectrum
$GW^{[n]}(X,\mathcal{L})$, $GW/\ell(X)$ for $GW^{[n]}(X,\mathcal{L}%
;\,\mathbb{Z}/\ell)$, $K(X)$ for $K^{[n]}(X,\mathcal{L})$, etc. We also
sometimes drop the parameter scheme $X$.

\begin{Preliminaries}
\textbf{The $C_{2}$-action on $K$-theory and homotopy fibrations.}%
\label{subsec:Prelim}
\end{Preliminaries}

Although it will not be needed in this paper, we note that the $C_{2}$-action
on $K^{[0]}(X,\mathcal{O}_{X})$ (resp.~$K^{[2]}(X,\mathcal{O}_{X})$) coincides
up to homotopy with the $C_{2}$-action defined in \cite{BKO1} and \cite{BKO}
(in the affine case) with the sign of symmetry $\varepsilon=1$
(resp.~$\varepsilon=-1)$.

Now suppose that the scheme $X$ has an ample family of line bundles, and
$\frac{1}{2}\in\Gamma(X,\mathcal{O}_{X})$. In \cite[Theorem 7.6]%
{Sch(myderived)}, the following are shown to hold. \vspace{1ex}

\begin{enumerate}
\item \label{item:HcartSqare} There is a homotopy fibration of spectra
\[
K^{\left[  n\right]  }(X,\mathcal{L})_{hC_{2}}\longrightarrow GW^{\left[
n\right]  }(X,\mathcal{L})\longrightarrow L^{\left[  n\right]  }%
(X,\mathcal{L}).
\]
The first term is the homotopy orbit spectrum for the $C_{2}$-action on the
$K$-theory spectrum $K^{\left[  n\right]  }(X,\mathcal{L})$. The homotopy
groups $\pi_{j}L^{\left[  n\right]  }(X,\mathcal{L})$ are naturally isomorphic
to Balmer's Witt groups $W^{n-j}(X,\mathcal{L})$ for all $n,j\in\mathbb{Z}$.
Recall from \cite{Balmer} that the groups $W^{i}$ are $4$-periodic in $i$ and
coincide with the classical Witt groups in degrees $\equiv0\ (\mathrm{mod\;}%
4)$. For a local ring $R$ with $\frac{1}{2}\in R$ we have $W^{i}(R)=0$ for
$i\not \equiv 0\ (\mathrm{mod\;}4)$.

\item \label{item:Fperiodic} There is a homotopy cartesian square of spectra
\[%
\begin{array}
[c]{ccc}%
GW^{\left[  n\right]  }(X,\mathcal{L}) & \longrightarrow & L^{\left[
n\right]  }(X,\mathcal{L})\\
\downarrow &  & \downarrow\\
K^{\left[  n\right]  }(X,\mathcal{L})^{hC_{2}} & \longrightarrow & \hat
{H}(C_{2},\,K^{\left[  n\right]  }(X,\mathcal{L}))
\end{array}
\]
where for a spectrum $Y$ with $C_{2}$-action, the Tate spectrum $\hat{H}%
(C_{2},Y)$ is the cofiber of the hypernorm map $Y_{hC_{2}}\rightarrow
Y^{hC_{2}}$; see \cite[Ch. 3]{Jardineetale}.

\item \label{item3:subsec:Prelim} Let $\eta\in GW_{-1}^{[-1]}(\mathbb{Z}%
[\frac{1}{2}])\cong W^{0}(\mathbb{Z}[\frac{1}{2}])$ correspond to the unit
$1\in W^{0}(\mathbb{Z}[\frac{1}{2}])$. Then the horizontal maps in
(\ref{item:Fperiodic}) induce equivalences of spectra
\[
GW^{\left[  n\right]  }[\eta^{-1}]\overset{\cong}{\longrightarrow}L^{\left[
n\right]  }\hspace{2ex}\text{and}\hspace{2ex}(K^{\left[  n\right]  })^{hC_{2}%
}[\eta^{-1}]\overset{\cong}{\longrightarrow}\hat{H}(C_{2},\,K^{\left[
n\right]  }).
\]
\noindent Both spectra $L^{\left[  n\right]  }(X,\mathcal{L})$ and $\hat
{H}(C_{2},\,K^{\left[  n\right]  }(X,\mathcal{L}))$ are $4$-periodic and the
map $L^{\left[  n\right]  }(X,\mathcal{L})\rightarrow\hat{H}(C_{2}%
,\,K^{\left[  n\right]  }(X,\mathcal{L}))$ commutes with the periodicity maps
by \cite[Remark 7.7]{Sch(myderived)}, \cite{WW00}. Hence the homotopy fiber
$\mathcal{F}$ of $GW^{\left[  n\right]  }(X,\mathcal{L})\rightarrow K^{\left[
n\right]  }(X,\mathcal{L})^{hC_{2}}$ satisfies
\[
\pi_{i}\mathcal{F}\cong\pi_{i+4}\mathcal{F}\text{\quad for all }i\in
\mathbb{Z}\text{.}%
\]

\end{enumerate}

\begin{Remarks}
For the affine non-connective analogs of the previous statements, see also
\cite{Kob99} and \cite[Theorem 13]{Wil05}. A possible generalization to
schemes goes via the Mayer-Vietoris principle \cite{Sch10(myMV)}. An alternate
approach is developed in \cite{Sch(myderived)}.
\end{Remarks}

\begin{Preliminaries}
\textbf{Presheaves of spectra} \label{subsec:PshSpXet}

Let $X$ be a scheme. Its small \'etale site $X_{\mathrm{\acute{e}t}}$ is
comprised of finite type \'etale $X$-schemes $U\to X$ and maps between
$X$-schemes, along with \'etale coverings. If $\operatorname{vcd}_{2}(X)=n$,
then $\operatorname{vcd}_{2}(U)\leq n$ for all $U \in X_{\mathrm{\acute{e}t}}$.

We denote by $\operatorname{PSp}(X_{\mathrm{\acute{e}t}})$ the model category
of presheaves of spectra on $X_{\mathrm{\acute{e}t}}$
\cite{Jardine:simplSpectra}. Its objects are contravariant functors from
$X_{\mathrm{\acute{e}t}}$ to spectra and maps are natural transformations of
such functors. We are mainly interested in the presheaves of spectra
$GW^{[n]}(\mathcal{L})$ sending $p:U\rightarrow X$ to $GW^{[n]}(U,p^{\ast
}\mathcal{L})$ and its $K$-theory analog. We shall often suppress
$\mathcal{L}$ and $[n]$ in the notation.

A map of presheaves of spectra $\mathcal{F}\rightarrow\mathcal{G}$ on
$X_{\mathrm{\acute{e}t}}$ is: \newline(1) a \emph{pointwise weak equivalence}
if for all $U\in X_{\mathrm{\acute{e}t}}$, the map $\mathcal{F}(U)\rightarrow
\mathcal{G}(U)$ is an equivalence of spectra; \newline(2) a \emph{local weak
equivalence} if for all points $x\in X$, $\mathcal{F}_{x}\rightarrow
\mathcal{G}_{x}$ is a weak equivalence of spectra, where $\mathcal{F}_{x}$ is
the filtered colimit $\mathcal{F}_{x}=\colim_{U\rightarrow X}\mathcal{F}(U)$
over all \'{e}tale neighborhoods $U$ of $x$; \newline(3) a \emph{cofibration}
if it is pointwise a cofibration, that is, if for all $U\in X_{\mathrm{\acute
{e}t}}$, the map $\mathcal{F}(U)\rightarrow\mathcal{G}(U)$ is a cofibration of
spectra in the sense of \cite{BousfieldFriedlander}; and \newline(4) a
\emph{local fibration} if it has the right lifting property with respect to
all cofibrations which are also local weak equivalences.

It is proved in \cite{Jardine:simplSpectra} that the category
$\operatorname{PSp}(X_{\mathrm{\acute{e}t}})$ together with the local weak
equivalences, cofibrations and local fibrations is a proper closed
(simplicial) model category.

Note that by Theorem \ref{rigidity} below, $(GW^{[n]}/\ell)_{x}\simeq
GW^{[n]}/\ell(k)$, where $k$ is a separable closure of the residue field of
$x$. However, there is, \textit{a priori}, no evident equivalence between
$(K^{hC_{2}}/\ell)_{x}$ and $K/\ell(k)^{hC_{2}}$ since the homotopy fixed
point functor $(\phantom{K})^{hC_{2}}$ does not commute with filtered
colimits, in general. Compare Lemma \ref{lem:etalHtpyLim} below.

From the theory of model categories, there exists a globally fibrant
replacement functor
\begin{equation}
\operatorname{PSp}(X_{\mathrm{\acute{e}t}})\longrightarrow\operatorname{PSp}%
(X_{\mathrm{\acute{e}t}}):\mathcal{F}\longmapsto\mathcal{F}^{\mathrm{\acute
{e}t}}. \label{eqn:etFibRpl}%
\end{equation}
By definition, this is a functor equipped with a natural local weak
equivalence $\mathcal{F}\rightarrow\mathcal{F}^{\mathrm{\acute{e}t}}$ for
which the map from $\mathcal{F}^{\mathrm{\acute{e}t}}$ to the final object is
a local fibration. The Hermitian Quillen-Lichtenbaum Theorem \ref{thm:QL} is a
statement about the map $\mathcal{F}\rightarrow\mathcal{F}^{\mathrm{\acute
{e}t}}$ when $\mathcal{F}$ is the Hermitian $K$-theory presheaf.

Call a square of presheaves of spectra \emph{pointwise homotopy cartesian} if
it becomes a homotopy cartesian square of spectra when evaluated at all finite
type \'{e}tale $X$-schemes. We need the following observations.
\end{Preliminaries}

\begin{lemma}
\label{lem:propsFet}

\begin{enumerate}
\item The globally fibrant replacement functor (\ref{eqn:etFibRpl}) sends
pointwise homotopy cartesian squares to pointwise homotopy cartesian squares.

\item Let $n$ be an integer. If a presheaf of spectra $\mathcal{F}$ satisfies
$\pi_{i}(\mathcal{F}_{x})=0$ for all $i\geq n$ and all points $x\in X$ then
$\pi_{i}(\mathcal{F}^{\mathrm{\acute{e}t}}(U))=0$ for all $i\geq n$ and $U\in
X_{\mathrm{\acute{e}t}}$.
\end{enumerate}
\end{lemma}

\begin{proof}
Both statements are true for any small Grothendieck site (with enough points
so that we can formulate the second part of the lemma). For the first part,
recall that in the category of spectra, homotopy cartesian is the same as
homotopy co-cartesian. Since cofibrations are pointwise cofibrations and
pointwise weak equivalences are local weak equivalences, it is clear that the
globally fibrant replacement functor preserves pointwise homotopy co-cartesian
squares. The second part is explicitly stated in \cite[Proposition
6.12]{Jardineetale}.
\end{proof}

There are evident analogs for the Nisnevich topology $X_{\mathrm{Nis}}$ on
$X$. Details are \textit{mutatis mutandis} the same. \medskip

For later reference we include the following results. Recall that $\ell
=2^{\nu}$.

\begin{lemma}
\label{lem:PureInsepExt} Let $F\subset L$ be a purely inseparable algebraic
extension of fields of characteristic $\neq2$. Then the inclusion $F\subset L$
induces equivalences of spectra
\[
K/\ell(F)\overset{\simeq}{\longrightarrow}K/\ell(L),\hspace{4ex}GW^{[n]}
/\ell(F)\overset{\simeq}{\longrightarrow}GW^{[n]}/\ell(L)
\]
and an isomorphism of Witt groups
\[
W(F)\overset{\cong}{\longrightarrow}W(L).
\]

\end{lemma}

\begin{proof}
The $K$-theory statement is due to Quillen \cite[Proposition 4.8]%
{quillen:higherI}. For Witt-groups, see \cite[p.~456, \S 2]{Arason:CohInv}.
The result for $GW^{[n]}/\ell$ now follows from the homotopy fibration
(\ref{subsec:Prelim} (\ref{item:HcartSqare})) and the vanishing of $W^{i}(k)$
for $k$ a field and $i\not \equiv 0\ (\mathrm{mod\;}4)$.
\end{proof}

\begin{theorem}
[Rigidity]\label{rigidity} Let $R$ be a Henselian local ring with residue
field $k$ and $\frac{1}{2}\in R$. Then the map $R\rightarrow k$ induces
equivalences of spectra
\[
K/\ell(R)\overset{\simeq}{\longrightarrow}K/\ell(k),\hspace{4ex}GW^{[n]}%
/\ell(R)\overset{\simeq}{\longrightarrow}GW^{[n]}/\ell(k)
\]
and an isomorphism of Witt groups
\[
W(R)\overset{\cong}{\longrightarrow}W(k).
\]

\end{theorem}

\begin{proof}
The $K$-theory (resp.~Witt-theory) result is due to Gabber
\cite{Gabber:Rigidity} (resp.~Knebusch \cite[Satz 3.3]{Knebusch:Isometrien}).
The claim for Grothendieck-Witt theory then follows from the homotopy
fibration (\ref{subsec:Prelim} (\ref{item:HcartSqare})).
\end{proof}

The following theorem is implicit in \cite{RO05} but was formulated only as an
equivalence on ($\operatorname{vcd}_{2}(X)-2$)-connected covers.

\begin{theorem}
[$K$-theoretic Quillen-Lichtenbaum]\label{thm:KQL} Let $X$ be a QL scheme.
Then for all $\nu\geq1$ the natural map
\[
K(X;\,\mathbb{Z}/2^{\nu})\longrightarrow K^{\mathrm{\acute{e}t}}%
(X;\,\mathbb{Z}/2^{\nu})
\]
is (\textit{$\mathrm{vcd}$}$_{2}(X)-2$)-coconnected.
\end{theorem}

\begin{proof}
With $\ell=2^{\nu}$ the map in the theorem factors as
\begin{equation}
K/\ell(X)\rightarrow(K/\ell)^{\mathrm{Nis}}(X)\rightarrow(K/\ell
)^{\mathrm{\acute{e}t}}(X) \label{eqn:KnisKet}%
\end{equation}
where $(K/\ell)^{\mathrm{Nis}}$ denotes a globally fibrant model for the
Nisnevich topology. We first show that the second map is
(\textit{$\mathrm{vcd}$}$_{2}(X)-2$)-coconnected. For fields, this is
\cite[(11), \S 5]{RO05}. The case of Henselian rings reduces to the case of
fields, by rigidity for $K$-theory and its \'{e}tale version (see
\emph{e.g.}~the proof of \cite[Lemma 4.14]{Mitchell:hyper} and
\cite[Proposition 6.1]{Mitchell}). For general $X$ as in the theorem, the
result follows from the Henselian case in view of the strongly convergent
Nisnevich descent spectral sequence applied to the homotopy fiber of the
second map in (\ref{eqn:KnisKet}).

To finish the proof, we note that the first map in (\ref{eqn:KnisKet}) is
always $0$-coconnected, so the theorem follows as soon as $\operatorname{vcd}%
_{2}(X)\geq2$. Since $\operatorname{dim}X\leq\operatorname{vcd}_{2}X$, we are
left with the cases $\operatorname{dim}X=0,1$. If $\operatorname{dim}X=0$ then
the first map in (\ref{eqn:KnisKet}) is an equivalence. If $\operatorname{dim}
X=1$ then this map is $(-1)$-coconnected. This assertion follows from the fact
that $K_{-1}$ is torsion free for noetherian schemes of Krull dimension
$\leq1$; see \cite[Lemma 2.5 (2)]{weibelDuke}. Therefore, the maps
$K/\ell\rightarrow\mathbb{K}/\ell$ and hence $(K/\ell)^{\mathrm{Nis}%
}\rightarrow(\mathbb{K}/\ell)^{\mathrm{Nis}}$ are $(-1)$-coconnected for such
schemes. Moreover, $\mathbb{K}/\ell\rightarrow(\mathbb{K}/\ell)^{\mathrm{Nis}%
}$ is a pointwise weak equivalence, by \cite{TT90}.
\end{proof}

\begin{lemma}
\label{W2tors} Suppose that $X$ is a quasi-compact scheme with an ample family
of line bundles. Then the following are equivalent.

\begin{enumerate}
\item[(1)] \label{W2torsA}\ There exists an integer $n>0$ such that
$2^{n}W(X)=0$.

\item[(2)] \label{W2torsB}\ $-1$ is a sum of squares in all the residue fields
of $X$.
\end{enumerate}
\end{lemma}

\begin{proof}
In \cite[Theorem 3, p. 189]{Knebusch} it is shown that (1) is equivalent to
the statement that all the residue fields of $X$ have $2$-primary torsion Witt
groups. The latter is equivalent to (2); see for instance \cite[Theorem
II.7.1]{Scharlau:Book}.
\end{proof}

\begin{remark}
If $X$ is affine, the condition that $-1$ is a sum of squares in all the
residue fields of $X$ is equivalent to $-1$ being a sum of squares in
$\Gamma(X,\mathcal{O}_{X})$; see for instance \cite[Proposition 4,
p.~190]{Knebusch}. If $X$ is non-affine, then $-1$ might be a sum of squares
in all residue fields without being a sum of squares in $\Gamma(X,\mathcal{O}%
_{X})$. Indeed, every smooth projective real curve $X$ with function field of
level $>1$ has the property that $-1$ is a sum of squares in all of its
residue fields, but not in $\Gamma(X,\mathcal{O}_{X})=\mathbb{R}$. For
example, take the closed subscheme of $\mathbb{P}_{\mathbb{R}}^{2}%
=\operatorname{Proj}(\mathbb{R}[X,Y,Z])$ cut out by the equation $X^{2}%
+Y^{2}+Z^{2}=0$.
\end{remark}

\medskip

\section{Proofs \label{Section 2: Proofs}}

Our proof of the following lemma uses the main result in \cite{HKO}.

\begin{lemma}
\label{lem:HtpyLimitFields} Theorem \ref{thm:Z2htpylimit} holds for fields $k$
with \textit{$\mathrm{vcd}$}$_{2}(k)<\infty$ and $\operatorname{char}(k)\neq2$.
\end{lemma}

\begin{proof}
We claim that the homotopy fiber $\mathcal{F}$ of $GW/\ell\rightarrow
K/\ell^{hC_{2}}$ is contractible. If $\operatorname{char}(k)=0$, this holds by
\cite{HKO}. If $\operatorname{char}(k)>0$, we reduce the claim to the case of
characteristic $0$ by means of \textquotedblleft Teichm{\"{u}}ller
lifting\textquotedblright.\ In effect, we may assume that $k$ is perfect,
since $K/\ell$ and $GW/\ell$ are invariant under purely inseparable algebraic
field extensions (Lemma \ref{lem:PureInsepExt}). Then the ring $V$ of
Witt-vectors over $k$ is a complete (hence Henselian) dvr with residue field
$k$ and fraction field $F$ of characteristic $0$. Furthermore,
\[
\mathrm{vcd}_{2}(F)\leq\mathrm{cd}_{2}(F)=\mathrm{cd}_{2}(k)+1=\mathrm{vcd}%
_{2}(k)+1<\infty\text{,}%
\]
by \cite[Expos\'e X, Th\'eor\`eme 2.2 (ii)]{SGA4Tome3LNM305} and \cite[II
\S 4.1]{Serre:cohGal5iemEdition}.

Let $\pi\in V$ be a uniformizer. We claim that $\alpha\colon V[t,t^{-1}
]\rightarrow F:\ t\mapsto\pi$, induces an equivalence $\mathcal{F}%
(\alpha):\mathcal{F}(V[t,t^{-1}])\overset{\simeq}{\rightarrow}\mathcal{F}(F)$.
It is known that $K/\ell(\alpha)$ is an equivalence (one may use the same
argument as for Witt groups below), and hence $K/\ell^{hC_{2}}(\alpha)$ and
$K/\ell_{hC_{2}}(\alpha)$ are equivalences. To show that $GW/\ell(\alpha)$ is
an equivalence, we consider the spectrum $L$ defined by~the homotopy fibration
$K_{hC_{2}}\rightarrow GW\rightarrow L$ (\ref{subsec:Prelim}
(\ref{item:HcartSqare})). The groups $\pi_{i}L$ are $4$-periodic, trivial for
local rings in degrees $\not \equiv 0\ (\mathrm{mod\;}4)$ and homotopy
invariant for regular rings. Using the localization exact sequence for
$V[t]\rightarrow V[t,t^{-1}]$, we get $W^{i}(V[t,t^{-1}])=0$ for
$i\not \equiv 0\ (\mathrm{mod\;}4)$. Thus, it remains to check that
$W^{0}(V[t,t^{-1}])\rightarrow W^{0}(F)$ is an isomorphism. This follows by a
comparison of the localization exact sequences for $V[t]\rightarrow
V[t,t^{-1}]$ and $V\rightarrow F$, which reduces to a map between short exact
sequences:
\[%
\begin{array}
[c]{ccccccccc}%
0 & \rightarrow & W^{0}(V\left[  t\right]  ) & \longrightarrow &
W^{0}(V\left[  t,t^{-1}\right]  ) & \longrightarrow & W^{0}(V) & \rightarrow &
0\\
&  & \downarrow &  & \downarrow &  & \downarrow &  & \\
0 & \rightarrow & W^{0}(V) & \longrightarrow & W^{0}(F) & \longrightarrow &
W^{0}(k) & \rightarrow & 0.
\end{array}
\]
Theorem \ref{rigidity} shows that the right vertical map, induced by the
reduction map modulo $\pi$, is an isomorphism. The left vertical map is an
isomorphism by homotopy invariance of Witt-theory. It follows that
$W^{0}(\alpha)$ is an isomorphism, as claimed.

Because augmentation makes $\mathcal{F}(V)$ a retract of $\mathcal{F}(V\left[
t,t^{-1}\right]  )$, combining with the equivalence $\mathcal{F}%
(\alpha):\mathcal{F}(V[t,t^{-1}])\overset{\simeq}{\rightarrow}\mathcal{F}(F)$
makes $\mathcal{F}(V)$ also a retract of $\mathcal{F}(F)$. However, since
$\operatorname{char}(F)=0$ and \textit{$\mathrm{vcd}$}$_{2}(F)<\infty$, we
have $\mathcal{F}(F)\simeq\ast$, by \cite{HKO}; this now implies that
$\mathcal{F}(V)\simeq\ast$. Finally, by rigidity, $\mathcal{F}(V)\rightarrow
\mathcal{F}(k)$ is an equivalence. Hence, $\mathcal{F}(k)\simeq\ast$, as sought.
\end{proof}

\begin{Preliminaries}
\textbf{Bott elements}. \label{prelim:BottElms} Let $p=8m$ for $m\geq1$ an
integer, and let $\ell$ be the highest power of $2$ that divides $3^{4m}-1$;
that is, $\ell=\max\{2^{k}\,\vert\, 2^{k}\text{ divides }3^{4m}-1\}$. An
element $\beta\in\pi_{p}(S^{0};\,\mathbb{Z}/\ell)$ in the mod $\ell$ stable
stems is called a \emph{(positive) Bott element} if it maps to the reduction
mod $\ell$ of a generator of $KO_{p}\cong\mathbb{Z}$ under the unit map
$S^{0}\rightarrow KO$, where $KO$ denotes the Bott-periodic real topological
$K$-theory spectrum. We require Bott elements in our proof (Lemma
\ref{lem:etalHtpyLim}) of the \'{e}tale version of the Homotopy Fixed Point
Theorem \ref{thm:Z2htpylimit}. Bott elements for higher Grothendieck-Witt
theory and arbitrary coefficients were first constructed in \cite{BKO}. In
what follows we give a simpler construction that suffices for our purposes.
\end{Preliminaries}

\begin{lemma}
\label{lem:posBott} Let $\ell$ be as in Section \ref{prelim:BottElms}. Then
there is an element $\beta\in\pi_{p}(S^{0}/\ell)$ in the mod $\ell$ stable
stem whose image in $\pi_{p}(KO/\ell)$ under the unit map $S^{0}\rightarrow
KO$ is the reduction mod $\ell$ of a generator of $KO_{p}=\mathbb{Z}$.
\end{lemma}

\begin{proof}
The construction of $\beta$, essentially due to Quillen, is based on Adams'
work on the image $J(\pi_{\ast}O)\subseteq\pi_{\ast}(S^{0})$ of the
$J$-homomorphism \cite{adams:ImJIV}. Recall that the $2$-primary part
$J(\pi_{p-1}O)_{(2)}$ of $J(\pi_{p-1}O)$ is cyclic of order $\ell$. For a
spectrum $F$, write $F_{\mathrm{tor}}$ for the homotopy fiber of the
rationalization map $F\rightarrow F_{\mathbb{Q}}$, and note that the natural
map $F_{\mathrm{tor}}/\ell\rightarrow F/\ell$ is an equivalence. Further,
recall that $\pi_{\ast}(S_{\mathrm{tor}}^{0})\rightarrow\pi_{\ast}(S^{0})$ is
an isomorphism for $\ast>0$. So, the $J$-homomorphism has image in $\pi_{\ast
}(S_{\mathrm{tor}}^{0})$. Consider the commutative diagram
\[
\xymatrix{
& \pi_p(S^0_{\mathrm{tor}}/\ell) \ar[r] \ar[d]^{\delta} & \pi_p(KO_{\mathrm{tor}}/\ell) \ar[d]^{\delta} \\
J(\pi_{p-1}O)_{(2)}\cong\Z/\ell \hspace{1ex}\ar@{^(->}[r] \ar[ur] &
\pi_{p-1}(S^0_{\mathrm{tor}}) \ar[r] & \pi_{p-1}(KO_{\mathrm{tor}})}
\]
in which the diagonal map exists because of the long exact sequence of
homotopy groups associated with the fibration $S_{\mathrm{tor}}^{0}%
\overset{\cdot\ell}{\rightarrow}S_{\mathrm{tor}}^{0}\rightarrow
S_{\mathrm{tor}}^{0}/\ell$. In \cite[p.~183, \S 2]{Quillen:letterToMilnor},
Quillen shows that the composition of the lower two maps is injective. It
follows that the composition of the diagonal map with the upper horizontal map
in the previous and in the following commutative diagram is injective
\[
\xymatrix{
& \pi_p(S^0_{\mathrm{tor}}/\ell) \ar[r] \ar[d]^{\cong} & \pi_p(KO_{\mathrm{tor}}/\ell)  \ar[d]^{\cong} \\
J(\pi_{p-1}O)_{(2)}\cong \Z/\ell \ar[r] \ar[ur] &
\pi_{p}(S^0/\ell) \ar[r] & \pi_{p}(KO/\ell) \cong \Z/\ell. }
\]
Therefore, the composition of the two lower horizontal maps in the last
diagram is an injection of finite groups of the same order. Hence, this
composition is an isomorphism. In particular, the map $\pi_{p}(S^{0}%
/\ell)\rightarrow\pi_{p}(KO/\ell)$ is surjective, and we can lift the
generator mod $\ell$ of $KO_{p}$ to an element $\beta\in\pi_{p}(S^{0}/\ell)$.
\end{proof}

\begin{lemma}
\label{lem:betaeta:nilpotent} Let $k$ be a separably closed field of
characteristic $\mathrm{char}(k)\neq2$, and let $\eta\in GW_{-1}%
^{[-1]}(k)\cong W^{0}(k)$ correspond to the unit of the ring $W^{0}(k)$. Then
$\beta\eta^{p}=0$ in $GW(k;\,\mathbb{Z}/\ell)$. In particular, $L/\ell
(k)[\beta^{-1}]\simeq\ast$.
\end{lemma}

\begin{proof}
By \cite{Kar1983}, \cite{Kar1984} the ring spectrum $GW/\ell(k)[\beta^{-1}]$
is equivalent to $KO/\ell$ and the natural map $GW(k)/\ell\rightarrow
GW/\ell(k)[\beta^{-1}]$ is an equivalence on connective covers. Thus, it
suffices to check that $\beta\eta^{p}=0$ in $\pi_{0}(KO/\ell)$. In $KO/\ell$,
the elements $\beta$ and $\eta$ are reductions of integral classes. More
precisely, $\beta$ is the reduction mod $\ell$ of $b^{m}$, where $b\in
KO_{8}=\mathbb{Z}$ is a generator, and $\eta^{4}\in GW_{-4}^{[-4]}%
(\mathbb{C})$. However, the map
\[
\mathbb{Z}/2\cong GW_{-4}^{[-4]}(\mathbb{C})=GW_{-4}^{[0]}(\mathbb{C}%
)\longrightarrow KO_{-4}=\mathbb{Z}%
\]
is trivial. Consequently, $\eta^{p}=0$ in $KO_{-p}$ and hence $\beta\eta
^{p}=0$.

For the $L$-theory statement, recall that $L=GW[\eta^{-1}]$; see
(\ref{subsec:Prelim} (\ref{item3:subsec:Prelim})). Therefore, $\beta\eta
^{p}=0$ implies $L/\ell(k)[\beta^{-1}]\simeq\ast$.
\end{proof}

\begin{lemma}
\label{lem:etalHtpyLim} Let $\nu>0$ be an integer and $\ell=2^{\nu}$. Let $X$
be a QL scheme. Then the map
\[
GW^{\mathrm{\acute{e}t}}/\ell(X)\longrightarrow(K^{hC_{2}})^{\mathrm{\acute
{e}t}}/\ell(X)
\]
is an equivalence.
\end{lemma}

\begin{proof}
For an integer $\nu>0$, a map of spectra is an equivalence mod $2^{\nu}$ if
and only if it is an equivalence mod $2$. Therefore, we can assume $\ell$ to
be as in Section \ref{prelim:BottElms}, and we have Bott elements at our
disposal. Consider the commutative diagram
\[%
\begin{array}
[c]{ccc}%
(GW/\ell)^{\mathrm{\acute{e}t}}(X)\longrightarrow & (GW/\ell\lbrack\beta
^{-1}])^{\mathrm{\acute{e}t}}(X)\longrightarrow & (L/\ell\lbrack\beta
^{-1}])^{\mathrm{\acute{e}t}}(X)\\
\downarrow & \downarrow & \downarrow\\
(K/\ell^{hC_{2}})^{\mathrm{\acute{e}t}}(X)\longrightarrow & (K/\ell^{hC_{2}%
}[\beta^{-1}])^{\mathrm{\acute{e}t}}(X)\longrightarrow & (\hat{H}/\ell
\lbrack\beta^{-1}])^{\mathrm{\acute{e}t}}(X)
\end{array}
\]
in which the right-hand square is obtained from the homotopy cartesian square
(\ref{subsec:Prelim} (\ref{item:Fperiodic})) by reduction mod $\ell$,
inverting the positive Bott element constructed in Lemma \ref{lem:posBott},
and taking \'{e}tale globally fibrant replacements. All these operations
preserve (pointwise) homotopy cartesian squares. Thus, the right-hand square
in the diagram is homotopy cartesian. By Lemma \ref{lem:betaeta:nilpotent},
Theorem \ref{rigidity} and Subsection \ref{subsec:Prelim}
(\ref{item:HcartSqare}), the upper right corner of the diagram is trivial.
Since $\hat{H}/\ell\lbrack\beta^{-1}]$ is a module spectrum over
$L/\ell\lbrack\beta^{-1}]$, the lower right corner of the diagram is trivial
as well. Hence, the middle vertical arrow is an equivalence. In view of Lemma
\ref{lem:propsFet} and Theorem \ref{rigidity}, the upper left horizontal arrow
is an equivalence on connective covers since $GW/\ell(F)\rightarrow
GW/\ell(F)[\beta^{-1}]$ has this property for separably closed fields $F$. By
Lemma \ref{lem:propsFet}, the lower left horizontal map is an equivalence on
some connected cover, because, by the solution of the $K$-theoretic
Quillen-Lichtenbaum conjecture \cite{Thomason-descent} in the generality of
\cite{RO05}, the map $K/\ell^{hC_{2} }\rightarrow K/\ell^{hC_{2}}[\beta^{-1}]$
is a pointwise (hence local) weak equivalence on (\textit{$\mathrm{vcd}$}%
$_{2}(X)-2$)-connected covers. Hence, the fiber of the left vertical map has
trivial homotopy groups in high degrees. By periodicity (\ref{subsec:Prelim}
(\ref{item3:subsec:Prelim})), we are done.
\end{proof}

\begin{lemma}
\label{lem:QLFieldsToSchemes} If Theorem \ref{thm:QL} holds for the residue
fields of a QL scheme $X$, then the map (\ref{eqn:QL}) is $n$-coconnected for
some integer $n$.
\end{lemma}

\begin{proof}
By definition, $GW^{\mathrm{\acute{e}t}}/\ell$ satisfies Nisnevich descent. In
positive degrees the same holds for $GW/\ell$ \cite{Sch(myderived)}. More
precisely, $\mathbb{G}W$ satisfies Nisnevich descent \cite[Theorem
9.7]{Sch(myderived)} and the map $GW\rightarrow\mathbb{G}W$ is an equivalence
on connective covers \cite[Proposition 8.7 or Theorem 8.14]{Sch(myderived)}.
The map between the $E^{2}$-pages of the corresponding Nisnevich descent
spectral sequences takes the form
\[
H_{\mathrm{Nis}}^{p}(X;\,\tilde{\pi}_{q}(GW/\ell))\longrightarrow
H_{\mathrm{Nis}}^{p}(X;\,\tilde{\pi}_{q}(GW^{\mathrm{\acute{e}t}}%
/\ell))\text{.}%
\]
By rigidity, see Theorem \ref{rigidity}, the assumption shows that the
canonically induced map
\[
\tilde{\pi}_{q}(GW/\ell)\longrightarrow\tilde{\pi}_{q}(GW^{\mathrm{\acute{e}%
t}}/\ell)
\]
of Nisnevich sheaves is an isomorphism for $q\geq\mathrm{vcd}_{2}(X)-1$. The
result follows from the fact that $H_{\mathrm{Nis}}^{p}(X,A)=0$ for
$p>\operatorname{dim}X$ and $p<0$ and strong convergence of the descent
spectral sequences \cite[Theorem 7.58]{Jardineetale}.
\end{proof}

\noindent\textbf{Proofs of Theorems \ref{thm:Z2htpylimit} and \ref{thm:QL}.\ }
Consider the commutative diagram:
\begin{equation}%
\begin{array}
[c]{ccc}%
GW/\ell(X) & \longrightarrow & GW^{\mathrm{\acute{e}t}}/\ell(X)\\
\downarrow &  & \downarrow\\
\left[  K/\ell^{hC_{2}}\right]  (X) & \longrightarrow & \left[  K/\ell
^{hC_{2}}\right]  ^{\mathrm{\acute{e}t}}(X).
\end{array}
\label{eqn:diag1}%
\end{equation}
By the solution of the $K$-theoretic Quillen-Lichtenbaum conjecture (Theorem
\ref{thm:KQL}), the homotopy fiber of the lower horizontal map is
($\mathrm{vcd}_{2}(X)-2$)-coconnected. Lemma \ref{lem:etalHtpyLim} shows the
right vertical map is an equivalence, while Lemma \ref{lem:HtpyLimitFields}
shows the left vertical map is an equivalence for fields. This implies the
Hermitian Quillen-Lichtenbaum Theorem \ref{thm:QL} for fields. Using Lemma
\ref{lem:QLFieldsToSchemes}, we have that the upper horizontal map is an
isomorphism in high degrees. It follows that the homotopy fiber of the left
vertical map in (\ref{eqn:diag1}) is trivial in high degrees. By periodicity
(\ref{subsec:Prelim} (\ref{item3:subsec:Prelim})), the homotopy fiber has
trivial homotopy in all degrees. Thus the left vertical map in
(\ref{eqn:diag1}) is an equivalence. This proves the Homotopy Fixed Point
Theorem \ref{thm:Z2htpylimit}. Since both vertical maps are equivalences, the
homotopy fiber of the upper horizontal map has trivial homotopy groups in the
same range as the homotopy fiber of the lower horizontal map. This proves the
Hermitian Quillen-Lichtenbaum Theorem \ref{thm:QL} for schemes $X$.
\hfill$\Box\medskip$ \vspace{1ex}

\noindent\textbf{Proof of Theorem \ref{thm:IntegralHtpyLimit}.} By Lemma
\ref{W2tors}, we have $2^{m}W^{0}(X)=0$ for some $m>0$. The homotopy groups of
$L^{\left[  n\right]  }(X)$ and the Tate spectrum $\hat{H}(C_{2},\,K^{\left[
n\right]  }(X))$ acquire compatible actions by $W^{0}$; see \cite{WW00},
\cite[Remark 7.7]{Sch(myderived)}. It follows that the homotopy groups of the
fiber $\mathcal{F}(X)$ of $L^{\left[  n\right]  }(X)\rightarrow\hat{H}%
(C_{2},\,K^{\left[  n\right]  }(X))$ also admit such an action, and are
therefore annihilated by $2^{m}$. However, by the Homotopy Fixed Point Theorem
\ref{thm:Z2htpylimit}, the homotopy cofiber of multiplication by $2^{m}%
\colon\mathcal{F}(X)\rightarrow\mathcal{F}(X)$ is trivial; that is,
multiplication by $2^{m}$ is an isomorphism on the homotopy groups of
$\mathcal{F}(X)$. As we have just noticed, this is the zero map. Hence,
$\mathcal{F}(X)\simeq\ast$, and the map (\ref{WilliamsMap}) is an equivalence.
\hfill$\Box\smallskip$ \vspace{1ex}

\noindent\textbf{Proof of Corollary \ref{cor:LthTate}.} This follows from
Theorems \ref{thm:Z2htpylimit} and \ref{thm:IntegralHtpyLimit} in view of the
homotopy cartesian square \ref{subsec:Prelim} (2). \hfill$\Box\smallskip$
\vspace{1ex}

\noindent\textbf{Proof of Corollary \ref{rem:ConnectedVsNonconnected}.}
\label{pf:ConnectedVsNonconnected} This follows from the fact that the
diagram
\[%
\begin{array}
[c]{ccc}%
GW(X) & \longrightarrow & \mathbb{G}W(X)\\
\downarrow &  & \downarrow\\
K(X)^{hC_{2}} & \longrightarrow & \mathbb{K}(X)^{hC_{2}}%
\end{array}
\]
is a homotopy cartesian square; see \cite[Theorem 8.14]{Sch(myderived)}%
.\hfill$\Box\smallskip$ \vspace{1ex}

If $\mathcal{L}=\mathcal{O}_{X}$ then the converse of the Integral Homotopy
Fixed Point Theorem \ref{thm:IntegralHtpyLimit} holds.

\begin{proposition}
\label{prop:IntegralConverse} Let $X$ be a scheme with an ample family of line
bundles and $\frac{1}{2}\in\Gamma(X,\mathcal{O}_{X})$. If the map
(\ref{WilliamsMap}) is an equivalence for $\mathcal{L}=\mathcal{O}_{X}$ then
no residue field of $X$ is formally real. More generally, this conclusion
holds if we assume only that the map (\ref{WilliamsMap}) is an equivalence
modulo some odd prime.
\end{proposition}

\begin{proof}
For any prime $q$, the map $GW^{\left[  n\right]  }/q(X)\rightarrow(K^{\left[
n\right]  }/q(X))^{hC_{2}}$ is an equivalence if the map (\ref{WilliamsMap})
is an integral equivalence. If $q$ is odd, then by (\ref{subsec:Prelim}
(\ref{item:Fperiodic})), the map $L^{[n]}/q(X)\rightarrow\hat{H}%
(C_{2},\,K^{[n]}/q(X))$ is an equivalence. Now multiplication by $2$ is an
equivalence on $K^{[n]}/q$, which implies that $\hat{H}(C_{2},\,K^{[n]}%
/q(X))\simeq\ast$, and hence $L^{[n]}/q(X)\simeq\ast$. Therefore, the Witt
ring $W(X)$ is a $\mathbb{Z}[\frac{1}{q}]$-algebra, since multiplication by
$q$ on $W(X)=L_{n}^{[n]}(X)$ is an isomorphism. If $X$ has a formally real
residue field $k$, then we obtain ring maps $\mathbb{Z}[\frac{1}%
{q}]\rightarrow W(X)\rightarrow W(k)\rightarrow W(\bar{k})=\mathbb{Z}$, where
$\bar{k}$ is a real closure of $k$, which leads to a contradiction.
\end{proof}

\begin{remark}
\label{Remark 3.5 11430} We should point out the necessity of our standing
assumption that $\frac{1}{2}\in\Gamma(X,\mathcal{O}_{X})$, although the cited
results in \cite{Sch10(myMV)} are proved in greater generality. Without this
assumption, the Homotopy Fixed Point Theorem \ref{thm:Z2htpylimit} cannot hold
for the following reason. As proved in \cite[\S \ 2]{Sch(myderived)}, the
fundamental theorem in Hermitian $K$-theory \cite{Kar1980} fails for the
$GW^{[n]}$-spectrum whenever $X$ has a residue field of characteristic $2$,
whereas it does hold for the $(K^{[n]})^{hC_{2}}$-spectrum; see the proof of
\cite[Theorem 6.2]{Sch(myderived)}. In particular, if $X$ has a residue field
of characteristic $2$ then (\ref{WilliamsMap}) is not an integral equivalence,
in general, even if $\operatorname{vcd}_{2}(X)<\infty$. Moreover, it is not a
$2$-adic equivalence for fields of characteristic $2$ (in this case the fiber
of (\ref{WilliamsMap}) is $2$-adically complete).

If $K$-theory of symmetric bilinear forms (that is, $GW$-spectra) is replaced
with $K$-theory of quadratic forms, then (\ref{WilliamsMap}) is not an
equivalence either, because the latter is not homotopy invariant for regular
rings, whereas $K$-theory and its homotopy fixed points are. In particular,
the quadratic analog of the map (\ref{WilliamsMap}) is not generally a
$2$-adic equivalence in characteristic $2$.
\end{remark}

\begin{remark}
\label{rem:oddprime} The odd-primary analog of the Hermitian
Quillen-Lichtenbaum Theorem \ref{thm:QL} can be read off from the isomorphisms
\cite[Remark 7.8]{Sch(myderived)}
\[
GW_{i}^{[n]}(X,\mathcal{L})\otimes\mathbb{Z}[1/2]\cong\left[  K_{i}%
^{[n]}(X,\mathcal{L})^{C_{2}}\otimes\mathbb{Z}[1/2]\right]  \oplus\left[
W^{n-i}(X)\otimes\mathbb{Z}[1/2]\right]  .
\]
Here, the $K$-summand may be computed by \'{e}tale techniques thanks to the
solution of the Bloch-Kato conjecture by Voevodsky, Rost, and others. Also,
$W^{r}$ denotes Balmer's Witt groups, which coincide up to $2$-torsion with
the higher Witt groups defined in \cite{Kar1980} (for affine schemes). On the
other hand, the odd-primary analog of the Homotopy Fixed Point Theorem is
false in general, even when $X$ is a $QL$ scheme: see Proposition
\ref{prop:IntegralConverse}.
\end{remark}

\begin{remark}
As for algebraic $K$-theory \cite{TT90}, higher Grothendieck-Witt theory can
be defined via perfect complexes instead of vector bundles. The techniques
employed in \cite{TT90} apply to $GW$-theory of perfect complexes along the
lines of \cite{Sch(myderived)}. With this definition, the results in this
section should remain valid without the \textquotedblleft ample family of line
bundles\textquotedblright\ assumption used in the definition of a $QL$ scheme.
To make this precise, one needs to display a strictly functorial model for
$GW$-theory of perfect complexes.
\end{remark}

\medskip

\section{Applications\label{Section 3: Applications}}



\begin{theorem}
\label{thm:GWalgKOtop} Let $X$ be a complex algebraic variety of (complex)
dimension $d$ which has an ample family of line bundles. Let $X_{\mathbb{C}}$
be the associated analytic topological space of complex points. Then for
$\ell=2^{\nu}>1$ and $n\in\mathbb{Z}$, the canonical map
\[
GW_{i}^{[n]}(X;\,\mathbb{Z}/\ell)\longrightarrow KO^{2n-i}(X_{\mathbb{C}%
};\,\mathbb{Z}/\ell)
\]
is an isomorphism for $i\geq d-1$ and a monomorphism for $i=d-2$.
\end{theorem}

\begin{proof}
The theories $GW^{[n]}$ have Bott-periodic topological counterparts
$GW_{\mathrm{top}}^{[n]}$ first explored in \cite{Karoubi:batelle} as
\[
\renewcommand\arraystretch{1.5}%
\begin{array}
[c]{ll}%
GW_{\mathrm{top}}^{[0]}(X_{\mathbb{C}})={_{1}\mathcal{L}}(X_{\mathbb{C}%
})=KO(X_{\mathbb{C}}), & GW_{\mathrm{top}}^{[-1]}(X_{\mathbb{C}}%
)={_{1}\mathcal{U}}(X_{\mathbb{C}})=\Omega^{2}KO(X_{\mathbb{C}}),\\
GW_{\mathrm{top}}^{[-2]}(X_{\mathbb{C}})={_{-1}\mathcal{L}}(X_{\mathbb{C}%
})=\Omega^{4}KO(X_{\mathbb{C}}), & GW_{\mathrm{top}}^{[-3]}(X_{\mathbb{C}%
})={_{-1}\mathcal{U}}(X_{\mathbb{C}})=\Omega^{6}KO(X_{\mathbb{C}}),
\end{array}
\]
which induce the maps
\[
GW^{[-n]}(X;\,\mathbb{Z}/\ell)\longrightarrow GW_{\mathrm{top}}^{[-n]}%
(X_{\mathbb{C}};\,\mathbb{Z}/\ell)=\Omega^{2n}KO(X_{\mathbb{C}};\,\mathbb{Z}%
/\ell)
\]
in the theorem. In the commutative diagram
\[%
\begin{array}
[c]{ccc}%
GW^{[n]}(X;\,\mathbb{Z}/\ell) & \longrightarrow & GW_{\mathrm{top}}%
^{[n]}(X_{\mathbb{C}};\,\mathbb{Z}/\ell)\\
\downarrow &  & \downarrow\\
\left[  K^{[n]}(X;\,\mathbb{Z}/\ell)\right]  ^{hC_{2}} & \longrightarrow &
\left[  KU^{[n]}(X_{\mathbb{C}};\,\mathbb{Z}/\ell)\right]  ^{hC_{2}},
\end{array}
\]
the lower horizontal map is $(d-2)$-coconnected, by a theorem of Voevodsky
\cite[Theorem 7.10]{VV}. By Theorem \ref{thm:IntegralHtpyLimit}, the left
vertical map is an (integral) equivalence. Finally, it is a classical theorem
that the right vertical map is also an (integral) equivalence. Indeed, for
$n=0$, this is the usual homotopy equivalence $KO\simeq KU^{hC_{2}}$ (see
\textsl{e.g.} \cite{Karoubidescent2001}); and for other $n\in\mathbb{Z}$, it
follows from the topological version of the fundamental theorem in Hermitian
$K$-theory \cite{Karoubi:batelle}.
\end{proof}

\begin{remark}
The proof shows that the theorem also holds for odd prime powers. We simply
need to remark that the map $K^{[n]}(X;\,\mathbb{Z}/\ell)\rightarrow
KU^{[n]}(X_{\mathbb{C}};\,\mathbb{Z}/\ell)$ is also ($d-2$)-coconnected for
odd prime powers $\ell$ due to the solution of the Bloch-Kato conjecture by
Voevodsky, Rost, Suslin and others. However, the odd prime analog of Theorem
\ref{thm:GWalgKOtop} can be more easily proved using Remark \ref{rem:oddprime}
in place of the Integral Homotopy Fixed Point Theorem. See also \cite{BKO} for
another argument in that case.
\end{remark}

\medskip

From now on, let $\ell$ again be a power of $2$. As a second application, we
give new and more conceptual proofs of the main results of \cite{BK} and
\cite{BKO1}. Let $\mathbb{Z}^{\prime}$ be short for $\mathbb{Z}[\frac{1}{2}]$.
Because $K^{\left[  n\right]  }$ has the same (nonequivariant) homotopy type
as $K$, from \cite{RW}, \cite[Corollary 8]{weibel:CRAS} we have the existence
of a homotopy cartesian square
\[%
\begin{array}
[c]{ccc}%
K^{\left[  n\right]  }(\mathbb{Z}^{\prime})/\ell & \longrightarrow &
K_{\mathrm{top}}^{\left[  n\right]  }(\mathbb{R})/\ell\\
\downarrow &  & \downarrow\\
K^{\left[  n\right]  }(\mathbb{F}_{3})/\ell & \longrightarrow &
K_{\mathrm{top}}^{\left[  n\right]  }(\mathbb{C})/\ell,
\end{array}
\]
where $K_{\mathrm{top}}$ stands for connective topological $K$-theory. Now,
since the fixed spectrum of $K^{\left[  n\right]  }$ is $GW^{\left[  n\right]
}$, the Homotopy Fixed Point Theorem \ref{thm:Z2htpylimit} applied to this
square yields the following.

\begin{theorem}
\label{thm:GW htpy cartesian square} For $\ell=2^{\nu}>1$ and $n\in\mathbb{Z}%
$, the square
\[%
\begin{array}
[c]{ccc}%
GW^{\left[  n\right]  }(\mathbb{Z}^{\prime})/\ell & \longrightarrow &
GW_{\mathrm{top}}^{\left[  n\right]  }(\mathbb{R})/\ell\\
\downarrow &  & \downarrow\\
GW^{\left[  n\right]  }(\mathbb{F}_{3})/\ell & \longrightarrow &
GW_{\mathrm{top}}^{\left[  n\right]  }(\mathbb{C})/\ell
\end{array}
\]
is homotopy cartesian on connective covers.\hfill$\Box$
\end{theorem}

\begin{remark}
According to \cite{Suslin}, these results do not depend on whether the fields
$\mathbb{R}$ and $\mathbb{C}$ are taken with the discrete or standard
Euclidean topology.
\end{remark}

\medskip

This theorem enables complete computation of the groups $GW^{\left[  n\right]
}(\mathbb{Z}^{\prime})$, up to finite groups of odd order (see \cite{BK}). In
particular, if $n=0$, the right vertical map in the above square can be
identified with the split surjective map $KO\times KO\rightarrow KO$ mod
$\ell$. Therefore, using $2$-adic completions we get the following corollary.

\begin{corollary}
For $i\geq0$ and any one-point space $\mathrm{pt}$, the natural map
\[
GW_{i}^{[0]}(\mathbb{Z}^{\prime})\longrightarrow GW_{i}^{[0]}(\mathbb{F}%
_{3})\oplus KO^{-i}(\mathrm{pt})
\]
is an isomorphism modulo finite groups of odd order.\hfill$\Box$
\end{corollary}

The groups $KO^{-i}(\mathrm{pt})$ are given by Bott periodicity, and the
groups $GW_{i}^{[0]}(\mathbb{F}_{3})$ were computed by Friedlander
\cite{Friedlander}. \vspace{1ex}

Similarly, let $F$ be a number field and $\mathcal{O}_{F}^{\prime}%
=\mathcal{O}_{F}[\frac{1}{2}]$ be its ring of $2$-integers. Assume that $F$ is
a $2$-regular totally real number field with $r$ real embeddings. Let $q$ be a
prime number such that the elements corresponding to the Adams operations
$\psi^{-1}$ and $\psi^{q}$ in the ring of operations of the periodic complex
topological $K$-theory spectrum generate the Galois group $F(\mu_{2^{\infty}%
})$ over $F$, where $F(\mu_{2^{\infty}})$ is obtained from $F$ by adjoining
all $2$-primary roots of unity. From \cite{HO} and \cite{Mitchell} we have a
homotopy cartesian square of connective spectra
\[%
\begin{array}
[c]{ccc}%
K^{\left[  n\right]  }(\mathcal{O}_{F}^{\prime})/\ell & \longrightarrow &
K_{\mathrm{top}}^{\left[  n\right]  }(\mathbb{R})^{r}/\ell\\
\downarrow &  & \downarrow\\
K^{\left[  n\right]  }(\mathbb{F}_{q})/\ell & \longrightarrow &
K_{\mathrm{top}}^{\left[  n\right]  }(\mathbb{C})^{r}/\ell.
\end{array}
\]
After application of the functor $(-)^{hC_{2}}$ to this homotopy cartesian
square, the Homotopy Fixed Point Theorem \ref{thm:Z2htpylimit} implies the
following result, which was first proved in \cite{BKO1} and which allows us to
compute completely the groups $GW_{i}^{\left[  n\right]  }(\mathcal{O}%
_{F}^{\prime})$ up to finite groups of odd order.

\begin{theorem}
\label{thm:ComputationOF} Let $\ell=2^{\nu}>1$ and $n\in\mathbb{Z}$. For a
$2$-regular totally real number field $F$ with $r$ real embeddings, the square
of spectra
\[%
\begin{array}
[c]{ccc}%
GW^{\left[  n\right]  }(\mathcal{O}_{F}^{\prime})/\ell & \longrightarrow &
GW_{\mathrm{top}}^{\left[  n\right]  }(\mathbb{R})^{r}/\ell\\
\downarrow &  & \downarrow\\
GW^{\left[  n\right]  }(\mathbb{F}_{q})/\ell & \longrightarrow &
GW_{\mathrm{top}}^{\left[  n\right]  }(\mathbb{C})^{r}/\ell
\end{array}
\]
is homotopy cartesian on connective covers. In particular, for $i\geq0$, $n=0$
and any one-point space $\mathrm{pt}$, the natural map
\[
GW_{i}^{[0]}(\mathcal{O}_{F}^{\prime})\longrightarrow GW_{i}^{[0]}%
(\mathbb{F}_{q})\oplus KO^{-i}(\mathrm{pt})^{r}.
\]
is an isomorphism modulo finite groups of odd order. \hfill$\Box$
\end{theorem}

\begin{remark}
Let $X$ be a $QL$ scheme. Our results also give the isomorphism
\begin{equation}
GW/\ell(X)[\beta^{-1}]\overset{\cong}{\longrightarrow}GW^{\mathrm{\acute{e}t}%
}/\ell(X)[\beta^{-1}]\label{eqn:GWBottinverted}%
\end{equation}
which was first proved in \cite{BKO}. Note that $GW^{\mathrm{\acute{e}t}}%
/\ell(X)\rightarrow GW^{\mathrm{\acute{e}t}}/\ell(X)[\beta^{-1}]$ is an
equivalence on connective covers. By the Homotopy Fixed Point Theorem
\ref{thm:Z2htpylimit}, cup product with the Bott element is an isomorphism in
high degrees, as the same is true for $K$-theory. However, this is also true
for \'{e}tale Hermitian $K$-theory, since by the Hermitian Quillen-Lichtenbaum
Theorem \ref{thm:QL} it coincides with Hermitian $K$-theory in high degrees.
Hence the equivalence (\ref{eqn:GWBottinverted}).
\end{remark}

\medskip

Let $A\bullet B$ denote an abelian group extension of $B$ by $A$, so that
there exists a short exact sequence
\[
0\rightarrow A\rightarrow A\bullet B\rightarrow B\rightarrow0.
\]
Another convention we follow is that $\mu_{2^{\nu}}^{\otimes i}$ denotes the
$i\,$th Tate twist of the sheaf of $2^{\nu}$\thinspace th roots of unity
$\mu_{2^{\nu}}$ (the kernel of multiplication by $2^{\nu}$ on the
multiplicative group scheme $\mathbb{G}_{m}$ over $\mathcal{O}_{F}^{\prime})$.
At one extreme, when $\nu=1$ this is independent of the Tate twist; at the
other, we use finiteness of the \'{e}tale cohomology groups of $\mathcal{O}%
_{F}^{\prime}$ to write $\mathbb{Z}_{2}^{\otimes j}$ for $\underleftarrow
{\mathrm{lim}}\,\mu_{2^{\nu}}^{\otimes j}$. \vspace{1ex}

By combining the Hermitian Quillen-Lichtenbaum Theorem \ref{thm:QL} with
\cite[Lemmas 6.12, 6.16]{BKO} we deduce our third computational application
for higher Grothendieck-Witt groups.

\begin{theorem}
\label{thm:2cd2computation} Suppose that $F$ is a totally imaginary number
field. The $2$-adically completed higher Grothendieck-Witt groups
$GW_{i}^{\left[  n\right]  }(\mathcal{O}_{F}^{\prime})_{\#}$ of the ring of
$2$-integers $\mathcal{O}_{F}^{\prime}$ of $F$ are computed in terms of
\'{e}tale cohomology groups as follows.\begin{table}[tbh]
\begin{center}%
\begin{tabular}
[c]{p{0.5in}|p{1.9in}|p{1.9in}|}\hline
$i\bmod8$ & $GW_{i}^{\left[  0\right]  }(\mathcal{O}_{F}^{\prime})_{\#}$ &
$GW_{i}^{\left[  1\right]  }(\mathcal{O}_{F}^{\prime})_{\#}$\\\hline
$8k>0$ & $H_{\mathrm{\acute{e}t}}^{2}(\mathcal{O}_{F}^{\prime},\mu_{2})\bullet
H_{\mathrm{\acute{e}t}}^{1}(\mathcal{O}_{F}^{\prime},\mu_{2})$ &
$H_{\mathrm{\acute{e}t}}^{2}(\mathcal{O}_{F}^{\prime},\mu_{2})\bullet
H_{\mathrm{\acute{e}t}}^{1}(\mathcal{O}_{F}^{\prime},\mathbb{Z}_{2}%
^{\otimes4k+1})$\\
$8k+1$ & $H_{\mathrm{\acute{e}t}}^{1}(\mathcal{O}_{F}^{\prime},\mu_{2})\bullet
H_{\mathrm{\acute{e}t}}^{0}(\mathcal{O}_{F}^{\prime},\mu_{2})$ &
$H_{\mathrm{\acute{e}t}}^{2}(\mathcal{O}_{F}^{\prime},\mu_{2})\bullet
H_{\mathrm{\acute{e}t}}^{1}(\mathcal{O}_{F}^{\prime},\mu_{2})$\\
$8k+2$ & $H_{\mathrm{\acute{e}t}}^{2}(\mathcal{O}_{F}^{\prime},\mathbb{Z}%
_{2}^{\otimes4k+2})\bullet H_{\mathrm{\acute{e}t}}^{0}(\mathcal{O}_{F}%
^{\prime},\mu_{2})$ & $H_{\mathrm{\acute{e}t}}^{1}(\mathcal{O}_{F}^{\prime
},\mu_{2})\bullet H_{\mathrm{\acute{e}t}}^{0}(\mathcal{O}_{F}^{\prime},\mu
_{2})$\\
$8k+3$ & $H_{\mathrm{\acute{e}t}}^{1}(\mathcal{O}_{F}^{\prime},\mathbb{Z}%
_{2}^{\otimes4k+2})$ & $H_{\mathrm{\acute{e}t}}^{2}(\mathcal{O}_{F}^{\prime
},\mathbb{Z}_{2}^{\otimes4k+3})\bullet H_{\mathrm{\acute{e}t}}^{0}%
(\mathcal{O}_{F}^{\prime},\mu_{2})$\\
$8k+4$ & $0$ & $H_{\mathrm{\acute{e}t}}^{1}(\mathcal{O}_{F}^{\prime
},\mathbb{Z}_{2}^{\otimes4k+3})$\\
$8k+5$ & $0$ & $0$\\
$8k+6$ & $H_{\mathrm{\acute{e}t}}^{2}(\mathcal{O}_{F}^{\prime},\mathbb{Z}%
_{2}^{\otimes4k+4})$ & $0$\\
$8k+7$ & $H_{\mathrm{\acute{e}t}}^{2}(\mathcal{O}_{F}^{\prime},\mu_{2})\bullet
H_{\mathrm{\acute{e}t}}^{1}(\mathcal{O}_{F}^{\prime},\mathbb{Z}_{2}%
^{\otimes4k+4})$ & $H_{\mathrm{\acute{e}t}}^{2}(\mathcal{O}_{F}^{\prime
},\mathbb{Z}_{2}^{\otimes4k+5})$\\\hline
\end{tabular}
\end{center}
\end{table}\begin{table}[tbh]
\begin{center}%
\begin{tabular}
[c]{p{0.5in}|p{1.9in}|p{1.9in}|}\hline
$i\bmod8$ & $GW_{i}^{\left[  2\right]  }(\mathcal{O}_{F}^{\prime})_{\#}$ &
$GW_{i}^{\left[  3\right]  }(\mathcal{O}_{F}^{\prime})_{\#}$\\\hline
$8k>0$ & $0$ & $H_{\mathrm{\acute{e}t}}^{1}(\mathcal{O}_{F}^{\prime
},\mathbb{Z}_{2}^{\otimes4k+1})$\\
$8k+1$ & $0$ & $0$\\
$8k+2$ & $H_{\mathrm{\acute{e}t}}^{2}(\mathcal{O}_{F}^{\prime},\mathbb{Z}%
_{2}^{\otimes4k+2})$ & $0$\\
$8k+3$ & $H_{\mathrm{\acute{e}t}}^{2}(\mathcal{O}_{F}^{\prime},\mu_{2})\bullet
H_{\mathrm{\acute{e}t}}^{1}(\mathcal{O}_{F}^{\prime},\mathbb{Z}_{2}%
^{\otimes4k+2})$ & $H_{\mathrm{\acute{e}t}}^{2}(\mathcal{O}_{F}^{\prime
},\mathbb{Z}_{2}^{\otimes4k+3})$\\
$8k+4$ & $H_{\mathrm{\acute{e}t}}^{2}(\mathcal{O}_{F}^{\prime},\mu_{2})\bullet
H_{\mathrm{\acute{e}t}}^{1}(\mathcal{O}_{F}^{\prime},\mu_{2})$ &
$H_{\mathrm{\acute{e}t}}^{2}(\mathcal{O}_{F}^{\prime},\mu_{2})\bullet
H_{\mathrm{\acute{e}t}}^{1}(\mathcal{O}_{F}^{\prime},\mathbb{Z}_{2}%
^{\otimes4k+3})$\\
$8k+5$ & $H_{\mathrm{\acute{e}t}}^{1}(\mathcal{O}_{F}^{\prime},\mu_{2})\bullet
H_{\mathrm{\acute{e}t}}^{0}(\mathcal{O}_{F}^{\prime},\mu_{2})$ &
$H_{\mathrm{\acute{e}t}}^{2}(\mathcal{O}_{F}^{\prime},\mu_{2})\bullet
H_{\mathrm{\acute{e}t}}^{1}(\mathcal{O}_{F}^{\prime},\mu_{2})$\\
$8k+6$ & $H_{\mathrm{\acute{e}t}}^{2}(\mathcal{O}_{F}^{\prime},\mathbb{Z}%
_{2}^{\otimes4k+4})\bullet H_{\mathrm{\acute{e}t}}^{0}(\mathcal{O}_{F}%
^{\prime},\mu_{2})$ & $H_{\mathrm{\acute{e}t}}^{1}(\mathcal{O}_{F}^{\prime
},\mu_{2})\bullet H_{\mathrm{\acute{e}t}}^{0}(\mathcal{O}_{F}^{\prime},\mu
_{2})$\\
$8k+7$ & $H_{\mathrm{\acute{e}t}}^{1}(\mathcal{O}_{F}^{\prime},\mathbb{Z}%
_{2}^{\otimes4k+4})$ & $H_{\mathrm{\acute{e}t}}^{2}(\mathcal{O}_{F}^{\prime
},\mathbb{Z}_{2}^{\otimes4k+5})\bullet H_{\mathrm{\acute{e}t}}^{0}%
(\mathcal{O}_{F}^{\prime},\mu_{2})$\\\hline
\end{tabular}
\end{center}
\end{table}\newline
\newpage

In particular, for $0\leq n\leq3$ and $i>0$, the group $GW_{i}^{\left[
n\right]  }(\mathcal{O}_{F}^{\prime})_{\#}$ is trivial when%
\[
\left\lfloor \frac{i-n}{2}\right\rfloor \equiv\left\lfloor \frac{4+n}%
{2}\right\rfloor \;(\mathrm{mod}\;4)\text{.}\
\vspace{-18pt}%
\]
\hfill$\Box$
\end{theorem}

\medskip

Comparing the above with \cite[Theorem 0.4]{RW}, one sees that the cohomology
terms involving twisted $\mathbb{Z}_{2}$-coefficients are detected by the
$K$-groups of $\mathcal{O}_{F}^{\prime}$.

\medskip The Lichtenbaum conjectures relate the orders of $K$-groups to values
of Dedekind zeta-functions of totally real number fields \cite{L1}, \cite{L2}.
We exhibit precise formulas relating the orders of higher Grothendieck-Witt
groups to values of Dedekind zeta-functions. If $m$ is even, let
$w_{m}=2^{a_{F}+\nu_{2}(m)}$ where $a_{F}:=(|\mu_{2^{\infty}}(F(\sqrt
{-1}))|)_{2}$ is the $2$-adic valuation and $2^{\nu_{2}(m)}$ is the
$2$-primary part of $m$. If $F=\mathbb{Q}(\zeta_{2^{b}}+\bar{\zeta}_{2^{b}})$,
then $a_{F}=b$; and when $F=\mathbb{Q}(\zeta_{r}+\bar{\zeta}_{r})$ (with $r$
odd)\textsf{\ }or $\mathbb{Q}(\sqrt{d})$ with $d>2$, then $a_{F}=2$. The
following theorem applies to these examples, where we write
$G_{\operatorname*{tor}}$ for the torsion subgroup of an abelian group $G$.
Its proof combines our results for $GW^{\left[  n\right]  }(\mathcal{O}%
_{F}^{\prime})$ with \cite[Theorem 0.2]{RW}.

\begin{theorem}
\label{thm:zetavalues} For every $2$-regular totally real abelian number field
$F$ with $r$ real embeddings, the Dedekind zeta-function of $F$ takes the
values
\[
\zeta_{F}(-1-4k)=\frac{\#GW_{8k+2}^{\left[  0\right]  }(\mathcal{O}%
_{F}^{\prime})}{2\#GW_{8k+3}^{\left[  0\right]  }(\mathcal{O}_{F}^{\prime}%
)}=2^{2r}\frac{\#GW_{8k+2}^{\left[  2\right]  }(\mathcal{O}_{F}^{\prime}%
)}{\#GW_{8k+3}^{\left[  2\right]  }(\mathcal{O}_{F}^{\prime})}=\frac{2^{r}%
}{w_{4k+2}}%
\]%
\[
\zeta_{F}(-3-4k)=2^{r}\frac{\#GW_{8k+6}^{\left[  0\right]  }(\mathcal{O}%
_{F}^{\prime})}{\#GW_{8k+7}^{\left[  0\right]  }(\mathcal{O}_{F}^{\prime}%
)}=2^{r}\frac{\#GW_{8k+6}^{\left[  2\right]  }(\mathcal{O}_{F}^{\prime
})_{\operatorname*{tor}}}{\#GW_{8k+7}^{\left[  2\right]  }(\mathcal{O}%
_{F}^{\prime})}=\frac{2^{r}}{w_{4k+4}}%
\]
up to odd multiples.\hfill$\Box\smallskip$
\end{theorem}

\begin{center}
\bigskip

\bigskip

A. Jon Berrick

Department of Mathematics, National University of Singapore, Singapore

Yale-NUS College, Singapore

e-mails: berrick@math.nus.edu.sg\qquad jon.berrick@yale-nus.edu.sg

\bigskip

Max Karoubi

UFR de math\'{e}matiques, Universit\'{e} Diderot Paris 7, France

e.mail : max.karoubi@gmail.com

\bigskip

Marco Schlichting

Mathematics Institute, University of Warwick, Coventry. United Kingdom

email: m.schlichting@warwick.ac.uk

\bigskip

Paul Arne {\O }stv{\ae }r

Department of Mathematics, University of Oslo, Norway

e.mail: paularne@math.uio.no
\end{center}

\end{document}